\documentclass[10pt]{amsart}
\usepackage[normalem]{ulem}%\uline,\uuline,\uwave,\sout,\xout
\usepackage[usenames]{color}
\usepackage{amsrefs}
\usepackage{amssymb}
\usepackage{pstricks-add}

 \newcommand{\R}{\mathbb R}
 \newcommand{\Z}{\mathbb Z}
\newcommand{\bs}{\backslash}\newcommand{\ep}{\varepsilon}

\DeclareMathOperator{\Int}{Int}

\theoremstyle{plain} \newtheorem{thm}{Theorem}
\newtheorem{cor}[thm]{Corollary} 
\newtheorem{lemma}[thm]{Lemma}

\theoremstyle{definition} \newtheorem{defn}[thm]{Definition}

\newtheorem{ex}[thm]{Example} 

\theoremstyle{remark} 

\DeclareMathOperator{\gr}{GR}
\DeclareMathOperator{\crec}{CR}
\newcommand{\scrwo}{\operatorname{SCR}}

\DeclareMathOperator{\mane}{M}
\newcommand{\prodset}[1]{{\mathcal{#1}}}
\DeclareMathOperator{\id}{Id}
\newcommand{\diag}{\Delta_X}
\DeclareMathOperator{\Fix}{Fix}
\DeclareMathOperator{\nw}{NW}

\newcommand{\tube}{\prodset{V}}

\DeclareMathOperator{\diam}{diam}

\newcommand{\projD}{\overline D}
\newcommand{\liftd}{D_{d_X,d_Y}}
\newcommand{\scrdf}{\scrwo_{d_X}(f)}
\newcommand{\scrdg}{\scrwo_{d_Y}(g)}
\newcommand{\mset}{\mathcal M}
\newcommand{\gm}{g_\mset}
\newcommand{\lyfg}{\Theta}
\newcommand{\lyf}{\theta}
\newcommand{\lyg}{\phi}
\DeclareMathOperator{\per}{Per}

\numberwithin{thm}{section}

\begin{document}

\title{Generalized Recurrence and the Nonwandering Set for Products}

  \author{Jim Wiseman}    \address{Agnes Scott College \\ Decatur, GA 30030} \email{jwiseman@agnesscott.edu}

\thanks{This work was supported by a grant from the Simons Foundation (282398, JW)}


\begin{abstract}
For continuous maps of compact metric spaces $f:X\to X$ and $g:Y\to Y$ and for various notions of topological recurrence, we study the relationship between recurrence for $f$ and $g$ and recurrence for the product map $f\times g:X\times Y \to X\times Y$.  For the generalized recurrent set $\gr$, we see that $\gr(f\times g)=\gr(f)\times\gr(g)$.  For the nonwandering set $\nw$, we see that $\nw(f\times g)\subset \nw(f)\times\nw(g)$ and give necessary and sufficient conditions on $f$ for equality for every $g$.  We also consider product recurrence for the chain recurrent set, the strong chain recurrent set, and the Ma\~n\'e set.
\end{abstract}

\maketitle

\section{Introduction}

Let  $f:X\to X$ and $g:Y \to Y$ be continuous maps of compact metric spaces.  We are interested in the relationship between recurrence for $f$ and $g$ and recurrence for the product map $f\times g:X\times Y \to X \times Y$, and in how that relationship varies depending on which notion of recurrence we consider.

The strongest notion of recurrence is periodicity.
It is clear that $\per(f\times g)$, the set of periodic points for $f\times g$, is equal to $\per(f)\times\per(g)$.
A slightly weaker condition is that a point is \emph{(positively) recurrent} if it is in its own $\omega$-limit set.  The question of whether the positive recurrent set of a given product is equal to the product of the positive recurrent sets
has been well studied and has led to some very deep and interesting mathematics; see \cite{ARec} and \cite{OZ} and the references therein.
In this paper, we consider the corresponding question for several less restrictive notions of recurrent set, most importantly the generalized recurrent set and the nonwandering set.

The interesting dynamics occurs on the nonwandering set, so in order to understand the relationship between the dynamics of a product map and the dynamics of the original maps, we need to understand the nonwandering set; we give necessary and sufficient conditions (Theorem~\ref{thm:prodnw}) for a point $x\in X$ to be product nonwandering, that is, for $(x,y)$ to be nonwandering  for $f\times g$ for any $g$ and any nonwandering point $y\in Y$.
Auslander's generalized recurrent set $\gr(f)$ (defined originally for flows (see \cite{Auslander}), and extended to maps  (see \cites{A,AA})) is a larger and in many ways more dynamically natural set, particularly for understanding Lyapunov functions; see \cite{FP} and the references in \cite{GRSCR}.  We show that $\gr(f\times g) = \gr(f)\times\gr(g)$ (Theorem~\ref{thm:main}).
The same is clearly true for the chain recurrent set, which reflects a still broader notion of recurrence.

We also consider product recurrence for Easton's strong chain recurrent set and Fathi and Pageault's Ma\~n\'e set.  These results come up for the most part in the study of the generalized recurrent set, but are also of independent interest.

In section~\ref{sect:defns}, we give definitions and background information for the various notions of recurrence and for metrics on the product space.  In section~\ref{sect:results}, we state and prove our results.

\section{Definitions and background}\label{sect:defns}

Throughout the paper, let $f:X\to X$ and $g:Y \to Y$ be continuous maps of compact metrizable spaces; unless stated otherwise, we will use the metrics $d_X$ and $d_Y$ respectively.    Let $B_d(x;\ep)$  be the closed $\ep$-ball around $x$, $B_d(x;\ep) =\{x'\in X : d(x,x')\le\ep\}$. 

\subsection{Recurrent sets}

\begin{defn}
A point $x \in X$ is \emph{nonwandering} for $f$ if for any neighborhood $U$ of $X$, there exists an $n>0$ such that $f^n(U) \cap U \ne \emptyset$.  We denote by $\nw(f)$ the set of nonwandering points.
\end{defn}

\begin{defn} 
An {\em $(\ep,f,d_X)$-chain} (or {\em $(\ep,d_X)$-chain}, if it is clear what the map is, or \emph{$\ep$-chain}, if the metric is also clear) of length $n$  from $x$ to $x'$ is a sequence $(x=x_0, x_1, \dots, x_n=x')$ such that $d_X(f(x_{i-1}),x_i)\le\ep$ for $i=1,\dots,n$.  A point $x$ is {\em chain recurrent} if for every $\ep>0$, there is an $\ep$-chain from $x$ to itself.  We denote by $\crec(f)$ the set of chain recurrent points.  (Chain recurrence is independent of the choice of metric; see, for example, \cite{Franks}.)
%Two points $x$ and $y$ in $\crec(f)$ are \emph{chain equivalent} if there are $\ep$-chains from $x$ to $y$ and from $y$ to $x$ for any $\ep>0$.
%The map $f$ is  \emph{chain transitive} on a subset $N$ of $X$ if for every $x,y\in N$ and every $\ep>0$, there is an $\ep$-chain from $x$ to $y$; the chain equivalence classes are called the \emph{chain transitive components}.
\end{defn}

%\begin{remark}
%Chain recurrence depends only on the topology, not on the choice of metric (see, for example, \cite{Franks}).
%\end{remark}

The following definition is due to Easton~\cite{E}.
\begin{defn} 
A {\em strong $(\ep,f,d_X)$-chain} (or {\em strong $(\ep,d_X)$-chain} or \emph{strong $\ep$-chain})  from $x$ to $x'$ is a sequence $(x=x_0, x_1, \dots, x_n=x')$ such that the sum of the errors is bounded by $\ep$, that is, $\sum_{i=1}^n d_X(f(x_{i-1}),x_i)\le\ep$.  A point $x$ is {\em $d_X$-strong chain recurrent} (or \emph{strong chain recurrent}) if for every $\ep>0$, there is a strong $(\ep,d_X)$-chain from $x$ to itself.  We denote the set of strong chain recurrent points by $\scrwo_{d_X}(f)$.  
%Two points $x$ and $y$ in $\screc(f)$ are \emph{$d$-strong chain equivalent} (or \emph{strong chain equivalent}) if there are strong $(\ep,d)$-chains from $x$ to $y$ and from $y$ to $x$ for any $\ep>0$.
%A subset $N$ of $X$ is \emph{$d$-strong chain transitive} (or \emph{strong chain transitive}) if  every $x$ and $y$ in $N$ are $d$-strong chain equivalent;  the strong chain equivalence classes are called the \emph{strong chain transitive components}.
\end{defn}

The strong chain recurrent set does depend on the choice of metric; see, for example, \cite{Y}.
One way to eliminate this dependence  is to take the intersection over all possible choices.  This leads to the following definition.

\begin{defn}[\cite{FP}]
The \emph{generalized recurrent set $\gr(f)$} is $\bigcap_{d_X'} \scrwo_{d_X'}(f)$, where the intersection is over all metrics $d_X'$ compatible with the topology of $X$.  
\end{defn}

We write $x_1\sim_f x_2$ if for any $\ep>0$ and any compatible metric $d_X'$ there is a strong $(\ep,f,d_X')$-chain from $x_1$ to $x_2$ and one from $x_2$ to $x_1$; then $\sim_f$ is a closed relation.

There are other, equivalent definitions of the generalized recurrent set; see \cite{FP,GRSCR,AA,A,Auslander}.  In particular, $\gr(f)$ was originally defined as the set of  points $x\in X$ such that all Lyapunov functions are constant on the orbit of $x$ \cite{AA,A,Auslander}; see section~\ref{sect:results}.

Another way to eliminate the dependence of the strong chain recurrent set on the choice of metric is to take the union over all possible choices.

\begin{defn}[\cite{FP}]
The \emph{Ma\~n\'e set $\mane(f)$} is $\bigcup_{d_X'} \scrwo_{d_X'}(f)$, where the  union is over all metrics $d_X'$ compatible with the topology of $X$.  
\end{defn}

We will need an equivalent definition of the Ma\~n\'e set.
We begin with some notation.  Let  $\diag$ be the diagonal in $X\times X$, $\diag=\{(x,x):x \in X\}$.  
   Let $\tube_{d_X}(\ep)$ (or $\tube(\ep)$) be the closed $\ep$-neighborhood of the diagonal $\diag$ in $X\times X$, $\tube_{d_X}(\ep) = \{(x_1,x_2) : d_X(x_1,x_2)\le\ep\}$.
%  , and $\opentube_d(\ep)$ (or $\opentube(\ep)$)  the open $\ep$-neighborhood, $\opentube_d(\ep) = \{(x_1,x_2) : d(x_1,x_2)<\ep\}$.
  
  For $\prodset{N}\subset X\times X$, 
  we denote by $\prodset{N}^n$ the $n$-fold composition of $\prodset{N}$ with itself, $\prodset{N}\circ\prodset{N}\cdots\circ\prodset{N}$, that is, 
  \begin{align*}
  \prodset{N}^n = & \{(x,x') : \text{there exists $z_0=x,z_1,\ldots,z_n=x'\in X$} \\ & \text{ such that $(z_{i-1},z_i)\in\prodset{N}$ for $i=1,\ldots,n$}\}.
  \end{align*}

\begin{defn}\label{defn:nchain}
Let $\prodset{N}$ be a neighborhood of $\diag$.  An \emph{$(\prodset{N},f)$-chain} (or simply \emph{$\prodset{N}$-chain} if the map is clear) from $x$ to $x'$ is a sequence of points $(x=x_0, x_1, \dots, x_n=x')$ in $X$ such that $(f(x_{i-1}),x_i) \in \prodset{N}$ for $i=1,\ldots,n$.
\end{defn}

Thus $(x,x')\in \prodset{N}^n$ exactly when there is an $(\prodset{N},\id)$-chain of length $n$ from $x$ to $x'$, where $\id$ is the identity map.

\begin{thm}[\cite{GRSCR}*{Theorem~3.3}]\label{thm:altdefmane}
A point $x$ is in $\mane(f)$ if and only if for any closed neighborhood $\prodset{D}$ of the diagonal in $X\times X$, there exist a closed symmetric neighborhood $\prodset{N}$ of the diagonal and an integer $n>0$ such that $\prodset{N}^{3^n} \subset \prodset{D}$ and there is an $(\prodset{N},f)$-chain of length $n$ from $x$ to itself.
\end{thm}

Any nonwandering point is clearly strong chain recurrent for any metric $d_X$, and any strong $\ep$-chain is clearly an $\ep$-chain, so we have the inclusions $\nw(f) \subset \gr(f) \subset \scrwo_{d_X}(f) \subset \mane(f) \subset \crec(f)$.

\subsection{Metrics on the product space}
Given two metrizable spaces $X$ and $Y$, we give the product space $X\times Y$ the product topology.  We will need to be able to go from metrics on $X$ and $Y$ to a metric on $X\times Y$, and vice versa.  There are many well-known ways of doing the former, all essentially equivalent; for convenience, we use the following definition.

\begin{defn}
 Let $d_X$ and $d_Y$ be metrics on $X$ and $Y$, respectively.  Denote by $\liftd$ the metric on $X\times Y$ given by $\liftd((x_1,y_1),(x_2,y_2))=d_X(x_1,x_2) + d_Y(y_1,y_2)$,
and observe that $\liftd$ induces the product topology on $X\times Y$.
\end{defn}

There does not seem to be much in the literature about   producing  metrics on $X$ and $Y$ from a metric on $X\times Y$, so we give the following construction.

\begin{defn}
Let $D$ be a metric on $X\times Y$.  Define a metric $\projD_X$ on $X$ by $\projD_X(x_1,x_2)= \max_{y\in Y}D((x_1,y),(x_2,y))$ and a metric $\projD_Y$ on $Y$ by $\projD_Y(y_1,y_2) = \max_{x\in X}D((x,y_1),(x,y_2))$.
\end{defn}

\begin{lemma}
$\projD_X$ is a metric on $X$, and if $D$ is compatible with the product topology on $X\times Y$, then $\projD_X$ is compatible with the topology on $X$.  The corresponding statements hold for $\projD_Y$.  Furthermore, $D((x_1,y_1),(x_2,y_2)) \le \projD_X(x_1,x_2) + \projD_Y(y_1,y_2)$ for all $x_1,x_2\in X$ and $y_1,y_2\in Y$.
\end{lemma}

\begin{proof}

We first show that $\projD_X$ is a metric.
It is clear from the definition that $\projD_X(x_1,x_2) \ge 0$, with $\projD_X(x_1,x_2)=0$ if and only if $x_1=x_2$,
and that $\projD_X(x_1,x_2) = \projD_X(x_2,x_1)$.  To prove the triangle inequality, observe that 
%$\projD_X(x_1,x_3) =  \max_{y\in Y}D((x_1,y),(x_3,y)) \le \max_{y\in Y} \left( D((x_1,y),(x_2,y)) + D((x_2,y),(x_3,y))  \right) \le \left(\max_{y\in Y} D((x_1,y),(x_2,y)) \right)   + \left(\max_{y\in Y} D((x_2,y),(x_3,y)) \right) = \projD_X(x_1,x_2) + \projD_X(x_2,x_3)$.
\begin{align*}
\projD_X(x_1,x_3)& =  \max_{y\in Y}D((x_1,y),(x_3,y)) \\
& \le  \max_{y\in Y} \left( D((x_1,y),(x_2,y)) + D((x_2,y),(x_3,y))  \right) \\
& \le   \left(\max_{y\in Y} D((x_1,y),(x_2,y)) \right)   + \left(\max_{y\in Y} D((x_2,y),(x_3,y)) \right) \\
& = \projD_X(x_1,x_2) + \projD_X(x_2,x_3).
\end{align*}

Next we show that $\projD_X$ is compatible with the topology on $X$.  Let $d_X$ be a compatible metric on $X$ and $d_Y$ a compatible metric on $Y$.  Take any point $x_0\in X$.  Since $D$ is compatible with the product topology, the function $F_y(x):=D((x_0,y),(x,y))$ is continuous for any $y\in Y$.  Thus the function $x\mapsto \projD_X(x_0,x) = \max_{y\in Y} F_y(x)$ is continuous by the Berge maximum theorem \cite{Berge}.  So for any $\ep>0$, there exists a $\delta>0$ such that if $d_X(x_0,x') < \delta$, then $|\projD_X(x_0,x_0) - \projD_X(x_0,x')| = \projD_X(x_0,x') < \ep$.  Thus $B_{d_X}(x_0;\delta) \subset B_{\projD_X}(x_0;\ep)$; that is, any $\projD_X$-ball has a $d_X$-ball inside it.

To prove the opposite inclusion, recall that the Hausdorff metric  induced by $D$ on the set of nonempty compact subsets of $X\times Y$   is given by $H_{D}(A,B) = \max\{\max_{b\in B}\{D(b,A)\},\max_{a\in A}\{D(a,B)\}\}$.
Observe that $H_D(\{x_1\}\times Y,\{x_2\}\times Y) \le \projD_X(x_1,x_2)$.
Since $X\times Y$ is compact, the topology induced by 
the Hausdorff metric for $D$ is the same as that induced by the Hausdorff metric for $d_X\times d_Y$ \cite{MR0042109}.   
Thus the map from $\{\{x\}\times Y:x\in X\}$, with the metric induced by $H_D$, to $X$ given by the projection $\{x\}\times Y \mapsto x$ is continous.  So for any  $\ep>0$ and any $x_0\in X$, there exists a $\delta>0$ such that if $H_D(\{x_0\}\times Y,\{x\}\times Y) < \delta$, then $d_X(x_0,x)<\ep$.  Since $H_D(\{x_0\}\times Y,\{x\}\times Y)\le \projD_X(x_0,x)$, we have that $B_{\projD_X}(x_0;\delta) \subset B_{d_X}(x_0;\ep)$.  Thus $\projD_X$ and $d_X$ induce the same topology on $X$.

Finally, we have that 
%$D((x_1,y_1),(x_2,y_2)) \le D((x_1,y_1),(x_2,y_1)) + D((x_2,y_1),(x_2,y_2))
% \le \projD_X(x_1,x_2) + \projD_Y(y_1,y_2)$.
 \begin{align*}
D((x_1,y_1),(x_2,y_2)) & \le D((x_1,y_1),(x_2,y_1)) + D((x_2,y_1),(x_2,y_2)) \\
& \le \projD_X(x_1,x_2) + \projD_Y(y_1,y_2).
\end{align*}

\end{proof}

\section{Recurrence for product maps}\label{sect:results}

Our main result concerns the relationships, for the various notions of recurrent set, between those of $f$ and $g$ and that of the product map $f\times g$.  The result for the chain recurrent set is easy, and the proof is included for completeness and to highlight the difference between $\ep$-chains and strong $\ep$-chains.  The result for the nonwandering set is a consequence of examples in the literature (see below), although I have not been able to find an explicit statement elsewhere.

\begin{thm}\label{thm:main}
Let $f:X\to X$ and $g:Y\to Y$ be continuous maps of compact metrizable spaces.  Let $d_X$ and $d_Y$ be compatible metrics on $X$ and $Y$,  respectively, and let $D$ be a metric on $X\times Y$ compatible with the product topology.  Then
\begin{enumerate}
	\item $\nw(f\times g) \subset \nw(f) \times\nw(g)$, and the inclusion can be strict. \label{item:nw}
	\item $\gr(f\times g) = \gr(f)\times\gr(g)$. \label{item:gr}
	\item $\mane(f\times g) \supset \mane(f)\times\mane(g)$, and the inclusion can be strict. \label{item:mane}
	\item $\scrwo_{\liftd}(f\times g) = \scrwo_{d_X}(f)\times\scrwo_{d_Y}(g)$. \label{item:scr-prod}
	\item $\scrwo_{D}(f\times g) \supset \scrwo_{\projD_X}(f)\times\scrwo_{\projD_Y}(g)$, and the inclusion can be strict. \label{item:scr-proj}
	\item $\crec(f\times g)=\crec(f)\times\crec(g)$. \label{item:crec}
\end{enumerate}
\end{thm}

\begin{proof}[Proof of Theorem~\ref{thm:main}(\ref{item:nw})]
It is easy to see that $\nw(f\times g) \subset \nw(f) \times\nw(g)$.  Let $(x,y)$ be a nonwandering point for $f\times g$, and let $U$ be any neighborhood of $x$ in $X$ and $V$ any neighborhood of $y$ in $Y$.  Then $U\times V$ is a neighborhood of $(x,y)$ in $X\times Y$, and so $(f\times g)^n(U\times V) \cap U\times V \ne \emptyset$ for some $n>0$; thus $f^n(U)\cap U \ne \emptyset$ and $g^n(V)\cap V \ne \emptyset$.  So $x$ is in $\nw(f)$ and $y$ is in $\nw(g)$.

Next we show that the inclusion can be strict.  Sawada \cite{Saw} and Coven and Nitecki \cite{CN} give  examples of  continuous maps $f$ of compact metric spaces such that $\nw(f^2)$ is strictly contained in $\nw(f)$.  Let $f$ be such a map, and let $x$ be a point in $\nw(f)\backslash\nw(f^2)$.  Since $x\notin \nw(f^2)$, there exists a neighborhood $U$ of $x$ such that $f^{2m}(U) \cap U = \emptyset$ for all $m>0$.  Define $g:\{0,1\}\to\{0,1\}$ by $g(0)=1$, $g(1)=0$.  Then $U \times \{0\}$ is a neighborhood of $(x,0)$, but $(f\times g)^n(U \times \{0\}) \cap U \times \{0\} = \emptyset$ for all $n\ge1$, since $g^n(0)=1$ if $n$ is odd and $f^n(U)\cap U=\emptyset$ if $n$ is even.  Thus $(x,0)$ is not in $\nw(f\times g)$, even though $x\in\nw(f)$ and clearly $0\in\nw(g)$.  

(Note that for any $k>0$, Lemma~\ref{lem:nwnotprod} gives examples such that $\nw(f^k)$ is strictly contained in $\nw(f)$.
There are also  functions constructed in \cite{GIM} that can be adapted to give an example for which $\nw(f\times g) \subsetneq \nw(f)\times\nw(g)$.)

\end{proof}

\begin{proof}[Proof of Theorem~\ref{thm:main}(\ref{item:crec})]

Recall that chain recurrence is independent of the choice of metric.  If $((x_0,y_0),(x_1,y_1),\ldots,(x_m,y_m))$ is an $(\ep,f\times g,\liftd)$-chain, then $(x_0,\ldots,x_m)$ is an $(\ep,f,d_X)$-chain and $(y_0,\ldots,y_m)$ is an $(\ep,g,d_Y)$-chain; thus $\crec(f\times g)\subset \crec(f)\times\crec(g)$.  
Now let $(x_0,\ldots,x_m)$ be an $(\ep/2,f,d_X)$-chain and $(y_0,\ldots,y_n)$ an $(\ep/2,g,d_Y)$-chain.  If $m \ne n$, we can concatenate the first chain with itself $n$ times and the second with itself $m$ times to get $\ep/2$-chains of equal lengths, so we may assume that $m=n$.  Then  $((x_0,y_0),(x_1,y_1),\ldots,(x_m,y_m))$ is an $(\ep,f\times g,\liftd)$-chain.
\end{proof}

The key observation in the preceding proof is that the concatenation of an $\ep$-chain from $p$ to $q$ with an $\ep$-chain from $q$ to $r$ gives an $\ep$-chain from $p$ to $r$, and, more generally, the concatenation of an arbitrary number of such $\ep$-chains gives an $\ep$-chain.  This is not true for strong $\ep$-chains:  Concatenating $N$ strong $\ep$-chains gives a $N\ep$-chain, so in order to control the sum of the errors in the concatenated chain, we must know in advance the number of chains to be concatenated.  That is why we need Lemma~\ref{lem:visits} below.

\begin{proof}[Proof of Theorem~\ref{thm:main}(\ref{item:scr-prod})]
It is easy to see that $\scrwo_{\liftd}(f\times g) \subset \scrwo_{d_X}(f)\times\scrwo_{d_Y}(g)$.  Take any $(x,y)\in\scrwo_{\liftd}(f\times g)$ and any $\ep>0$, and let $((x_0,y_0)=(x,y), (x_1,y_1),\ldots,(x_n,y_n)=(x,y))$ be a strong $(\ep,f\times g,\liftd)$-chain from $(x,y)$ to itself.  
Then $\sum_{i=1}^n d_X(f(x_{i-1}), x_i) \le \sum_{i=1}^n d_X(f(x_{i-1}), x_i) + d_Y(g(y_{i-1}),y_i) = \sum_{i=1}^n \liftd((f(x_{i-1}),g(y_{i-1})),(x_{i},y_i)) \le\ep$, so $(x_0=x, x_1,\ldots,x_n=x)$ is a strong $(\ep,f,d_X)$-chain from $x$ to itself.  Since $\ep$ was arbitrary, we have $x\in\scrwo_{d_X}(f)$.  Similarly, we have $y\in\scrwo_{d_Y}(g)$, and so $(x,y)\in\scrwo_{d_X}(f)\times\scrwo_{d_Y}(g)$.

Next we show that $\scrwo_{d_X}(f)\times\scrwo_{d_Y}(g) \subset \scrwo_{\liftd}(f\times g)$.  We must show that for $x\in\scrdf$, $y\in\scrwo_{d_Y}(g)$, and $\ep>0$, there are strong $\ep$-chains from $x$ to $x$ and $y$ to $y$ of the same length.  We use the following lemmas.

\begin{lemma}\label{lem:visits}
For any $x\in\scrdf$ and any $\ep>0$, there exists an $n=n(x,\ep)>0$ such that for any $x'$ with $x\sim_f x'$, there is a strong $\ep$-chain $(x_0=x,x_1,\ldots,x_{n(x,\ep)}=x)$ of length $n(x,\ep)$ from $x$ to itself passing through $x'$, that is, such that $x_i=x'$ for some $i$.
\end{lemma}

\begin{proof}[Proof of Lemma~\ref{lem:visits}]
Restrict $f$ to the space $X_x=\{x':x\sim_f x'\}$.  Since $X_x$ is compact metric, it is separable; let $\{s_1,s_2,\ldots\}$ be a countable dense subset of $X_x$.  For each $i$, choose $\delta_i\le\frac\ep{3\cdot2^{i}}$ such that if $d_X(z,w)\le\delta_i$, then $d_X(f(z),f(w))\le \frac\ep{3\cdot2^{i}}$.
Since $x\sim_f s_i$, there is a strong $\delta_i/2$-chain from $x$ to $s_i$, and one from $s_i$ to $x$.  Concatenate these chains to get a strong $\delta_i$-chain from $x$ to itself passing through $s_i$; call this chain $L_i$.  If $d_X(x',s_i)\le\delta_i$, then we can substitute $x'$ for $s_i$ in the chain to get a strong $\frac\ep{2^i}$-chain from $x$ to itself passing through $x'$ of the same length; call this chain $L_i(x')$.

Pick a finite subcover of the open cover $\{B_{\delta_i}(s_i)\}_{i=1}^\infty$ of $X_x$; by renumbering if necessary, we can assume that the subcover is $\{B_{\delta_i}(s_i)\}_{i=1}^N$.  Then $x'$ is in $B_{\delta_i}(s_i)$ for some $i$.  Concatenate the chains $L_1, L_2, \ldots, L_{i-1}, L_i(x'),L_{i+1},\ldots, L_N$ to get a chain from $x$ to itself passing through $x'$.  Since each $L_j$ is a strong $\frac\ep{2^j}$-chain, and $\sum_{j=1}^N\frac\ep{2^j}<\ep$, the concatenation is a strong $\ep$-chain.  The length of the chain is the sum of the lengths of the chains $L_j$, which is independent of $x'$, so define $n(x,\ep)=\sum_{j=1}^N \text{length}(L_j)$.

\end{proof}

For any $x\in\scrdf$ and $\ep>0$, define the set $N(x,\ep)$ to be $\{n:$ for every $x'$ with $x\sim_f x'$, there is a strong $\ep$-chain of length $n$ from $x$ to itself passing through $x'\}$.  A set of natural numbers is an \emph{IP-set} \cite{FM} if it consists of all finite sums of some infinite set.

\begin{lemma}\label{lem:IP}
For any $x\in\scrdf$ and $\ep>0$,  the set $N(x,\ep)$ contains an IP-set.
\end{lemma}

\begin{proof}[Proof of Lemma~\ref{lem:IP}]
Observe that $N(x,\ep)$ contains all finite sums of the set $\{n(x,\frac{\ep}{2^i})\}_{i=1}^\infty$, with $n(x,\frac{\ep}{2^i})$ as in Lemma~\ref{lem:visits}, since a concatenation of strong $\frac\ep{2^i}$-chains (with distinct $i$'s) is a strong $\ep$-chain.
If $X_x$ does not consist of a single periodic orbit, then necessarily $n(x,\frac{\ep}{2^i})\to\infty$ as $i\to\infty$, so the set $\{n(x,\frac{\ep}{2^i})\}_{i=1}^\infty$ is infinite.  If  $X_x$ is a single periodic orbit, $\{x, f(x),\dots,f^p(x)=x\}$, then we can choose $n(x,\frac{\ep}{2^i}) = ip$, and again the set $\{n(x,\frac{\ep}{2^i})\}_{i=1}^\infty$ is infinite.
\end{proof}

Take any point $(x,y)$ in $\scrdf\times\scrdg$ and any $\ep>0$, and note that Lemmas~\ref{lem:visits} and \ref{lem:IP} apply to $y$ and $g$ as well.  Consider $g$ restricted to the set $Y_y=\{y': y\sim_g y'\}$.  Since $N(x,\ep/4)$ contains an IP-set, Theorem~12 of \cite{FM} guarantees that there exist a $y'\in Y_y$ and an $m_{y} \in N(x,\ep/4)$ such that $d_Y(y', g^{m_{y}}(y'))\le\ep/4$ (because $N(x,\ep/4)$ is a set of topological recurrence).  Thus $(y', g(y'),g^2(y'),\ldots, g^{m_{y}-1}(y'),y')$ is a strong $(\ep/4,g)$-chain of length $m_y$ from $y'$ to itself.  Similarly, there exist an $x'\in X_x$ and an $m_x\in N(y,\ep/4)$ such that $(x', f(x'),\ldots,f^{m_x-1}(x'),x')$ is a strong $(\ep/4,f)$-chain.

Since $m_y\in N(x,\ep/4)$, there is a strong $(\ep/4, f)$-chain $(x_0=x, x_1,\ldots,x_i=x',\ldots,x_{m_y}=x)$ of length $m_y$ from $x$ to itself passing through $x'$.  Similarly, there is a strong $(\ep/4, g)$-chain $(y_0=x, y_1,\ldots,y_j=y',\ldots,y_{m_x}=y)$ of length $m_x$ from $y$ to itself passing through $y'$.  Then $$L_x=(x_0=x, x_1,\ldots, x_i=x', f(x'),\ldots,f^{m_x-1}(x'),x', x_{i+1},\ldots,x_{m_y}=x)$$ is a strong $(\ep/2,f)$-chain of length $m_x+m_y$ from $x$ to itself, and similarly $L_y=(y_0=y, y_1,\ldots, y_j=y', g(y'),\ldots,g^{m_y-1}(y'),y', y_{j+1},\ldots,y_{m_x}=y)$ is a strong $(\ep/2,g)$-chain of length $m_x+m_y$ from $y$ to itself.  Thus the product of the chains $L_x$ and $L_y$ (that is, the chain $((x_0,y_0)=(x,y), (x_1,y_1),\ldots,(x,y))$)
is a strong $(\ep,f\times g, \liftd)$-chain of length $m_x+m_y$ from $(x,y)$ to itself.  Since $\ep$ was arbitrary, we have $(x,y)\in\scrwo_{\liftd}(f\times g)$.

\end{proof}

\begin{proof}[Proof of Theorem~\ref{thm:main}(\ref{item:scr-proj})]

Let $(x,y)$ be a point in $\scrwo_{\projD_X}(f)\times\scrwo_{\projD_Y}(g)$.  As in the proof of Theorem~\ref{thm:main}(\ref{item:scr-prod}), for any $\ep>0$ we can construct a strong $(\ep/2,f,\projD_X)$-chain $(x_0=x,x_1,\ldots,x_n=x)$ and a strong $(\ep/2,f,\projD_Y)$-chain $(y_0=y,y_1,\ldots,y_n=y)$ of the same length.  Since for all $i$, $1\le i\le n$, we have $D((f(x_{i-1}),g(y_{i-1})), (x_i,y_i)) \le \projD_X(f(x_{i-1}),x_i)+ \projD_Y(g(y_{i-1}),y_i)$, the product chain $((x_0,y_0)=(x,y),\ldots,(x_n,y_n)=(x,y))$ is a strong $(\ep,f\times g, D)$-chain.  Since $\ep$ was arbitrary, we have $(x,y)\in \scrwo_{D}(f\times g)$.

We prove that the inclusion can be strict by contradiction.  Assume that $\scrwo_{D}(f\times g) = \scrwo_{\projD_X}(f)\times\scrwo_{\projD_Y}(g)$ for every metric $D$ on $X\times Y$ compatible with the product topology, and let $(x,y)$ be a point in $\mane(f\times g)$. Thus there is a metric $D$ for which $(x,y)\in \scrwo_{D}(f\times g)$, which implies that $(x,y) \in \scrwo_{\projD_X}(f)\times\scrwo_{\projD_Y}(g)$, which implies that $(x,y)\in\mane(f)\times\mane(g)$. But in Example~\ref{ex:strictm} below, we construct $f$ and $g$ such that $ \mane(f)\times\mane(g) \subsetneq \mane(f\times g)$, which means that the inclusion must be strict for some metric on $X\times Y$.  (In fact, the metric derived from the Minkowski ?-function discussed in the example will work.)

\end{proof}

\begin{proof}[Proof of Theorem~\ref{thm:main}(\ref{item:gr})]

This follows from Theorem~\ref{thm:main}(\ref{item:scr-prod}) and (\ref{item:scr-proj}).  We have 
%$\gr(f\times g) = \cap_D \scrwo_{D}(f\times g) \supset \cap_D \scrwo_{\projD_X}(f)\times\scrwo_{\projD_Y}(g) \supset \gr(f)\times\gr(g)$,
\begin{align*}
\gr(f)\times\gr(g) &\subset \cap_D \scrwo_{\projD_X}(f)\times\scrwo_{\projD_Y}(g)  \\
& \subset  \cap_D \scrwo_{D}(f\times g) \\
& = \gr(f\times g),
%
%&= \cap_D \scrwo_{D}(f\times g) \\
%& \supset \\
%& \supset ,
\end{align*}
 where the intersection is over all metrics $D$ on $X\times Y$ compatible with the product topology, and 
 %$\gr(f)\times\gr(g) = \cap_{d_X,d_Y} \scrwo_{d_X}(f)\times\scrwo_{d_Y}(g) = \cap_{d_X,d_Y} \scrwo_{\liftd}(f\times g) \supset \gr(f\times g)$,
 \begin{align*}
 \gr(f\times g) & \subset \cap_{d_X,d_Y} \scrwo_{\liftd}(f\times g)\\
  & = \cap_{d_X,d_Y} \scrwo_{d_X}(f)\times\scrwo_{d_Y}(g)\\
  & = \gr(f)\times\gr(g),
 %) &= 
%& =
%& \supset ,
\end{align*}
 where the intersection is over all metrics $d_X$ and $d_Y$ compatible with the topologies on $X$ and $Y$, respectively.

\end{proof}

To prove that the inclusion in Theorem~\ref{thm:main}(\ref{item:mane}) can be strict, we need an alternative description of the Ma\~n\'e set $\mane(f)$, in terms of ordinary chain recurrence.

\begin{defn}
We say that $x$ is \emph{chain-recurrent through $A$} if for every $\ep>0$, there is an $\ep$-chain from $x$ to itself lying entirely in $A$.  We denote by $\crec_A(f)$ the set of points chain-recurrent through $A$.
\end{defn}

  Fathi and Pageault originally proved the following result for homeomorphisms (\cite{FP}*{Theorem~3.5}).  Here, we extend the result to continuous maps and give a somewhat more topological, less technical proof.

%\begin{thm}[{\cite[Thm.~3.5]{FP}}]\label{thm:mfpt}
\begin{thm}\label{thm:mfpt}

$\mane(f)= \Fix(f) \cup \crec_{X\backslash\Int(\Fix(f))}(f)$.
\end{thm}

\begin{proof}

We first prove that $\mane(f) \subset \Fix(f) \cup \crec_{X\backslash\Int(\Fix(f))}(f)$.  Take any point $x$ not in $\Fix(f) \cup \crec_{X\backslash\Int(\Fix(f))}(f)$ and any metric $d'$ compatible with the topology on $X$.  We will show that $x\not\in\scrwo_{d'}(f)$; since $d'$ was arbitrary, this implies that $x\not\in\mane(f)$.

First, we show that there exist an $\ep_0>0$ and an $\alpha>0$ such that for any $\ep\le\ep_0$, any (ordinary, not strong) ($\ep,d')$-chain from $x$ to itself must pass through the  $\alpha$-interior of $\Fix(f)$.  (The $\alpha$-interior of $\Fix(f)$ is $\{z: B_{d'}(z;\alpha)\subset\Fix(f)\}$,  which we will denote by $\Int_\alpha(\Fix(f))$.)  Assume not.  Then for any $\alpha>0$, there exists a sequence $\{\ep_j\}$ tending to 0 such that for each $j$ there is an $\ep_j$-chain from $x$ to itself which stays entirely within $\alpha$ of $X\bs\Int(\Fix(f))$. For a given $\ep$, choose $\delta>0$  such that if $d'(x',x'') < \delta$, then $d'(f(x'),f(x''))<\ep/3$; then choose $\alpha < \min(\ep/3,\delta)$ and $\ep_j<\alpha$ and let $(x_0=x, x_1, \ldots, x_n=x)$ be an $\ep_j$-chain which stays entirely within $\alpha$ of $X\bs\Int(\Fix(f))$.  For each $x_i, 0<i<n$, there is a point $\tilde x_i \in X\bs\Int(\Fix(f))$ with $d'(x_i,\tilde x_i)<\alpha$; then $(x_0=x, \tilde x_1, \ldots,\tilde x_{n-1}, x_n=x)$ is an $\ep$-chain that stays entirely in $X\bs\Int(\Fix(f))$.  Since $\ep$ was arbitrary, we have $x\in\crec_{X\backslash\Int(\Fix(f))}(f)$, contrary to assumption.

Thus there exist an $\ep_0>0$ and an $\alpha>0$ such that for any $\ep\le\ep_0$, any  $\ep$-chain from $x$ to itself must pass through  $\Int_\alpha(\Fix(f))$.
Now choose an $\ep<\min(\ep_0,\alpha)$ and let $(x_0=x, x_1, \dots, x_n=x)$ be an $\ep$-chain from $x$ to itelf.  Let $m$ be the smallest index such that $x_m\in\Int_\alpha(\Fix(f))$, and let $m+k$ be the smallest index greater than $m$ such that $x_{m+k}\not\in\Fix(f)$.  Then $f(x_i)=x_i$ for $i=m,\ldots,m+k-1$, so we have that $\sum_{i=1}^{n} d'(f(x_{i-1}),x_i) \ge \sum_{i=m+1}^{m+k} d'(f(x_{i-1}),x_i) = \sum_{i=m+1}^{m+k} d'(x_{i-1},x_i) \ge d'(x_m,x_{m+k}) \ge \alpha$.  Thus $x\not\in\scrwo_{d'}(f)$, as desired.

Next we show that $ \Fix(f) \cup \crec_{X\backslash\Int(\Fix(f))}(f) \subset \mane(f)$.  
Let $d$ be a metric on $X$ compatible with the topology.
It is obvious that any fixed point is in $\mane(f)$, so assume that
for any $\ep>0$, there is an $(\ep,f,d)$-chain from $x$ to itself lying in ${X\backslash\Int(\Fix(f))}$.  We will use the alternative definition for $\mane(f)$ from Theorem~\ref{thm:altdefmane}.  Let $\prodset D$ be a
closed neighborhood  of the diagonal; we must find  a closed symmetric neighborhood $\prodset{N}$ of the diagonal and an integer $n>0$ such that $\prodset{N}^{3^n} \subset \prodset{D}$ and there is an $(\prodset{N},f)$-chain of length $n$ from $x$ to itself.  
Choose $\ep$ such that $\tube(4\ep) \subset \prodset D$, choose $\delta \le \frac\ep2$ such that if $d(z_1,z_2)\le\delta$, then $d(f(z_1),f(z_2))\le \frac\ep2$,
and let $(x_0=x, x_1,\dots, x_n=x)$ be a $\delta$-chain from $x$ to itself contained in $ X\backslash\Int(\Fix(f))$; by starting with a  $\frac\delta2$-chain and perturbing the points slightly if necessary, we can  assume that in fact $(x_0=x, x_1,\dots, x_n=x)$ is contained in $X\backslash \Fix(f)$.

For $1\le i \le n$, let $C_i = B_d(f(x_{i-1});\delta_C) \cup B_d(x_i;\delta_C)$, where  $\delta_C$ will be chosen later.
  To construct $\prodset N$, we will need to ensure that the $C_i$'s are pairwise disjoint; thus we need to ensure that for $i<j$, we have $x_i\ne x_j$, $x_i \ne f(x_{j-1})$, $f(x_{i-1}) \ne x_j$, and $f(x_{i-1}) \ne f(x_{j-1})$.  First, in the cases $x_i=x_j$, $f(x_{i-1}) = x_j$, and $f(x_{i-1}) = f(x_{j-1})$, we can shorten the chain to $(x_0,\dots,x_{i-1},x_j,x_{j+1}\dots,x_n)$ and still have a $\delta$-chain from $x_0=x$ to $x_n=x$.  
Now consider the case $x_i = f(x_{j-1})$; we must have $j>i+1$, since $x_{j-1}$ is not fixed.  Let $i'$ be the smallest index such that $x_i' = f(x_{j-1})$ for some $j>i+1$, and let $j'$ be the largest such $j$ for $i'$.  Then the shortened chain $(x_0=x,x_1,\ldots,x_{i'-1},x_{i'},x_{j'+1},x_{j'+2},\ldots,x_n=x)$ is an $\ep$-chain, since $d(f(x_{i'}),x_{j'+1}) \le d(f(x_{i'}),f(x_{j'})) + d(f(x_{j'}),x_{j'+1}) = d(f^2(x_{j'-1}),f(x_{j'})) + d(f(x_{j'}),x_{j'+1}) \le \frac\ep2 + \delta \le \ep$.  We may have to perform this truncation several times, but, by construction, the consecutive terms $x_{i'}$, $x_{j'+1}$ will remain, and thus we will end up with an $\ep$-chain, which we will still denote by  $(x_0=x, x_1,\dots, x_n=x)$.

Choose $\delta_C \le \ep$ small enough that the collection $\{ C_i : 1\le i \le n \}$ is pairwise disjoint, and choose $\ep_0 < \min_{i\ne j} d(C_i,C_j)$;
%$\ep_0$ less than the minimum distance between any two sets in $C^\ast$;
 we can assume that $\ep_0 \le \ep$ as well.  Now define $\prodset N = \tube(\frac{\ep_0}{3^n}) \bigcup (\bigcup_{i=1}^n C_i \times C_i)$.  Since $(f(x_{i-1}),x_i)$ is in $C_i\times C_i$, we have that $(x_0=x,x_1,\ldots,x_n=x)$ is an $(\prodset{N},f)$-chain.  To see that $\prodset{N}^{3^n} \subset \prodset{D}$, let $z_0, z_1,\ldots,z_{3^n}$ be a sequence with $(z_{j-1},z_j)\in\prodset{N}$ for $1\le j\le 3^n$; we want to show that $d(z_0,z_{3^n}) \le 4\ep$.  Since the $C_i$'s  are more than $\ep_0$ apart, there exists at most one $C_i$ such that there is a pair $(z_j,z_{j+1})$ in the chain in $C_i\times C_i$; thus any other two consecutive points must be within $\frac{\ep_0}{3^n}$ of each other.  So $d(z_0,z_{3^n}) \le 3^n\cdot \frac{\ep_0}{3^n} + \max\{\diam(C_i)\} \le \ep_0 + \ep + 2\delta_C \le 4\ep.$

\end{proof}

\begin{proof}[Proof of Theorem~\ref{thm:main}(\ref{item:mane})]

We first use Theorem~\ref{thm:mfpt} to show that $\mane(f\times g) \supset \mane(f)\times\mane(g)$.  (This also follows from Theorem~\ref{thm:main}(\ref{item:scr-prod}).)  Take any point $(x,y)\in \mane(f)\times\mane(g)$.  If either $x$ or $y$ is fixed, then $(x,y)$ is clearly in $\mane(f\times g)$, so assume that $x\in\crec_{X\backslash\Int(\Fix(f))}(f)$ and $y\in\crec_{Y\backslash\Int(\Fix(g))}(g)$.  For any $\ep>0$, there exist an $(\ep/2,f,d_X)$-chain $(x_0=x,\ldots,x_m=x)$ in ${X\backslash\Int(\Fix(f))}$ and an $(\ep/2,g,d_Y)$-chain $(y_0=y,\ldots,y_n=y)$ in ${Y\backslash\Int(\Fix(g))}$.  We may assume that $m=n$ (if not, concatenate the first chain with itself $n$ times and the second with itself $m$ times to get two chains of equal length).  Then the product chain $((x_0,y_0)=(x,y),\ldots,(x,y))$ is an $(\ep,f\times g,\liftd)$-chain in $(X\times Y)\backslash\Int(\Fix(f\times g))$.  Since $\ep$ was arbitrary, we have $(x,y)\in \mane(f\times g)$.

The following example shows that the inclusion can be strict.

\begin{ex}[\cite{GRSCR}*{Example~3.8}]\label{ex:strictm}
Let $X$ be the circle with the usual topology, and let $f:X\to X$ be a map fixing every point on the closed left semicircle and moving points on the open right semicircle counterclockwise.  Then by Theorem~\ref{thm:mfpt}, $\mane(f)$ is the closed left semicircle.  Let $Y$ be the set $\{0,1\}$ and $g:Y\to Y$ the permutation switching the two points; then $\mane(g)=Y$.  Since $f\times g$ has no fixed points, $\mane(f\times g) = \crec(f\times g) = X\times Y$, which strictly contains $\mane(f)\times\mane(g)$.  (One can show that if we define the metric $D$ on $X\times Y$ by giving $X\times\{0\}$ the usual circle metric and  $X\times\{1\}$ the usual metric on the left semicircle and the metric induced by the Minkowski ?-function on the right semicircle, we have $\scrwo_D(f\times g)=X\times Y$.)

\end{ex}

\end{proof}

The generalized recurrent set was originally defined in terms of Lyapunov functions, and Fathi and Pageault showed in \cite{FP} that the strong chain recurrent set can be defined in terms of Lipschitz Lyapunov functions, so we can restate some of our main results.
Let $\lyf:X\to\R$ be a Lyapunov function for $f$ (that is, $\lyf(f(x))\le \lyf(x)$ for all $x$).  Following the notation in \cite{FP}, denote by $N(\lyf)$ the set of neutral points, $N(\lyf) = \{x\in X:\lyf(f(x))=\lyf(x)\}$.  Let $L(f)$ be the set of continuous Lyapunov functions for $f$, and let $L_{d_X}(f)$ be the set of Lipschitz (with respect to the metric $d_X$) Lyapunov functions for $f$.
  Since $\gr(f)=\bigcap_{\lyf\in L(f)} N(\lyf)$ \cite{A,AA,Auslander} and $\scrwo_{d_X}(f)=\bigcap_{\lyf\in L_{d_X}(f)} N(\lyf)$ \cite{FP}, we have the following corollaries to Theorem~\ref{thm:main}.

\begin{cor}
Let $f:X\to X$ and $g:Y\to Y$ be continuous maps of compact metric spaces, with metric $d_X$ and $d_Y$, respectively. 
	\begin{enumerate}

	\item A point $(x,y)$ is in $\bigcap_{{\lyfg}\in L(f\times g)}N(\lyfg)$ if and only if $x$ is in $\bigcap_{\lyf\in L(f)} N(\lyf)$ and $y$ is in $\bigcap_{\lyg\in L(g)} N(\lyg)$.
	\item  A point $(x,y)$ is in $\bigcap_{{\lyfg}\in L_{\liftd}(f\times g)}N(\lyfg)$ if and only if $x$ is in $\bigcap_{\lyf\in L_{d_X}(f)} N(\lyf)$ and $y$ is in $\bigcap_{\lyg\in L_{d_Y}(g)} N(\lyg)$.

	\end{enumerate}
\end{cor}

Theorem~\ref{thm:main} also gives an easy proof of the following well-known result from number theory.  For real numbers $w$ and $z$, denote by $|w-z|_1$ the difference in their fractional parts; that is, $|w-z|_1 = |w-z| \pmod 1$.

\begin{cor}
Let $\alpha$ and $\beta$ be real numbers.  Then for any $\ep>0$, there are infinitely many positive integers $n$ such that $|n\alpha - n\beta|_1 < \ep$.  
\end{cor}

\begin{proof}
Let $S^1$ be the circle, considered as $\R/\Z$, with metric $d(x_1,x_2) = |x_1-x_2|_1$, and let $R_\theta:S^1\to S^1$ be rotation by some $\theta$, $R_\theta(x) = x+\theta \pmod 1$.    It is clear that $x$ is strong chain recurrent for $R_\theta$ for every $x$ and every $\theta$, and so, by Theorem~\ref{thm:main}(\ref{item:scr-prod}), $(0,0)\in\scrwo_{D_{d,d}}(R_\alpha\times R_\beta)$.
Let $((x_0,y_0)=(0,0),\ldots,(x_n,y_n)=(0,0))$ be a strong $(\ep,R_\alpha\times R_\beta,D_{d,d})$-chain; since $R_\alpha\times R_\beta$ is an isometry, we have that $D_{d,d}((x_n,y_n),(R_\alpha\times R_\beta)^n(x_0,y_0))\le\ep$.
Thus $\ep\ge D_{d,d}((R_\alpha\times R_\beta)^n(0,0),(0,0)) = |n\alpha|_1 + |n\beta|_1 \ge |n\alpha-n\beta|_1$.  Since this is true for any $\ep$, there must be infinitely many such $n$.
\end{proof}

A point $x$ is \emph{(positively) recurrent} for $f:X\to X$ if, for any neighborhood $U$ of $x$, $x$ returns to $U$, that is, $f^n(x)\in U$ for some  $n>0$.  Thus any recurrent point is nonwandering.  A recurrent point $x$ is \emph{product recurrent} if for any recurrent point $y$ of any map $g:Y\to Y$, the point $(x,y)$ is recurrent for the product map $f\times g$.  It is well known that a point is product recurrent if and only if it is distal \cite{Furst}*{Theorem~9.11}.
 We have a somewhat analogous result for nonwandering points.
 
 \begin{defn}
 Let $f:X\to X$ be a continuous map of a compact metric space.  A point $x$ in $X$ is \emph{product nonwandering} if, for any continuous map $g:Y\to Y$ of a compact metric space and any nonwandering point $y$ for $g$, the point $(x,y)$ is nonwandering for $f\times g$.  The map $f$ is \emph{locally (topologically) mixing at $x$} if for any neighborhood $U$ of $x$, there exists an integer $N$ such that $f^n(U)\cap U \ne\emptyset$ for all $n\ge N$.
 \end{defn}

Any topologically mixing map (such as the doubling map on the circle) is locally mixing at every point.  The disjoint union of two circles, with the map given by doubling on each circle, is not topologically mixing, but it is locally mixing at every point; the same is true for the identity map on any nontrivial space.  More generally, a map is locally mixing at any fixed point.

\begin{ex}\label{ex:spiral}
Let $X$ be the unit disk in $\R^2$ and let $f:X\to X$ be a map that fixes the center $(0,0)$ and the north pole $(0,1)$,  moves other interior points in a clockwise spiral out towards the boundary circle, and moves points on the circle clockwise toward the north pole.    Then the nonwandering set consists of the center and the boundary circle, and $f$ is locally mixing at every nonwandering point.
\end{ex}

\begin{thm}\label{thm:prodnw}
Let $f:X\to X$ be a continuous map of a compact metric space to itself.  A point $x$ in $X$ is product nonwandering if and only if $f$ is locally (topologically) mixing at $x$.  
\end{thm}

\begin{proof}
We prove the ``if'' direction first. Take any $y\in\nw(g)$ and any neighborhood $W$ of $(x,y)$ in $X\times Y$.  Let $U$ and $V$ be neighborhoods of $x$ in $X$ and $y$ in $Y$, respectively, such that $U\times V \subset W$, and choose $N$ such that  $f^n(U)\cap U \ne\emptyset$ for all $n\ge N$. The set of return times for $V$, $\{m>0:g^m(V)\cap V\ne\emptyset\}$, is infinite \cite{vries}*{Proposition~4.3.2}.  Thus there exists a return time $m\ge N$, so we have $(f\times g)^m(W)\cap W \supset (f\times g)^m(U\times V) \cap (U\times V) = (f^m(U) \cap U)\times (g^m(V) \cap V) \ne \emptyset$.  Since $W$ was arbitrary, $(x,y)$ is nonwandering for $f\times g$.

To prove the ``only if'' direction, we use the following lemma, which may be of independent interest.

\begin{lemma}\label{lem:nwnotprod}
Let $\mset$ be any infinite set of natural numbers.  Then there exist a continuous map $g_\mset:Y_\mset\to Y_\mset$ of a compact metric space to itself, a point $y\in \nw(g_\mset)$, and a neighborhood $V$ of $y$ such that $\{m>0:g_\mset^m(V)\cap V\ne\emptyset\} = \mset$.
\end{lemma}

\begin{proof}[Proof of Lemma~\ref{lem:nwnotprod}]

Let $\mset = \{m_i\}_{i=1}^\infty$, with $m_1<m_2<\dots$.  We construct the space $Y_\mset = \bigcup_{i=0}^\infty Y_i$ as a subset of $\R^2$, where the spaces $Y_i$ and the map $\gm$ are defined as follows; see Figure ~\ref{fig:gm} for an example with $m_1=3$, $m_2=5$, and $m_3=6$.

Define $Y_0 = \{(0,0)\}\cup \{(\frac1n,0) : 1\le n <\infty\}$, and define
\begin{eqnarray*}
\gm(\frac1n,0)  &= & (\frac1{n-1},0)\ (n>1) \\
\gm(1,0) &=& (0,0) \\
\gm(0,0) &=&(0,0).
\end{eqnarray*}

For $i\ge1$, define $Y_i = \{(\frac1n,\frac1{i}): 1 \le n \le m_i\}$, and define
\begin{eqnarray*}
\gm(\frac1n,\frac1{i})  &= & (\frac1{n-1},\frac1{i}) \ (n>2) \\
\gm(\frac12,\frac1{i})  &= & (1,0) \\
\gm(1,\frac1{i}) &=&(\frac1{m_i},\frac1{i}) 
\end{eqnarray*}
(unless $m_i=1$, which can happen only if $i=1$; then $\gm(1,1)=(1,0)$).

The set $Y_\mset$ is compact and the map $\gm$ is continuous.  Observe that $\gm^k(1,0)=(0,0)$ for $k\ge1$, and that $\gm^{m_i}(1,\frac1i)=(1,0)$.  The point $y=(1,0)$ is nonwandering, because every neighborhood of $y$ contains a point of the form $(1,\frac1i)$.  The set $V =  \{y\} \cup \{(1,\frac1i): 1 \le i<\infty\}$ is a neighborhood of $y$ such that $\{m>0:g_\mset^m(V)\cap V\ne\emptyset\} = \mset$.

\end{proof}

\begin{figure}
\begin{pspicture}(-1,-0.5)(12,5.5)
\parametricplot[plotstyle=dots,plotpoints=8]{1}{8}{10 t div 0}
%\parametricplot[plotstyle=dots,plotpoints=10]{10}{19}{5 t div 0}
\parametricplot[plotstyle=dots,plotpoints=3]{1}{3}{10 t div 5}
\parametricplot[plotstyle=dots,plotpoints=5]{1}{5}{10 t div 2.5}
\parametricplot[plotstyle=dots,plotpoints=6]{1}{6}{10 t div 1.67}
\psdots(0,0)
\rput(.7,0){$\cdots$}
\rput(10,.7){$\vdots$}
\rput(5,.7){$\vdots$}
\rput(3.33,.7){$\vdots$}
\rput(2.5,.7){$\vdots$}
\rput(2,.7){$\vdots$}
\rput(1.5,.7){$\vdots$}
\rput(.7,.7){$\vdots$}

\rput(11,0){$Y_0$}
\rput(11,5){$Y_1$}
\rput(11,2.5){$Y_2$}
\rput(11,1.67){$Y_3$}

\psset{linewidth=.5pt}
%\psline[linewidth=1pt,linearc=.25]{->}(5,0)(7.5,-1)(10,0)
\pnode(5,0){A}
\pnode(10,0){B} \psset{nodesep=3pt} \ncarc{->}{A}{B}
\pnode(0,0){C} \ncarc{->}{B}{C}
\pnode(3.33,0){D} \pnode(2.5,0){E}
\ncarc{->}{E}{D} \ncarc{->}{D}{A}
\nccircle{->}{C}{.2}

\pnode(3.33,5){A1} \pnode(5,5){B1} \pnode(10,5){C1}
\ncarc{->}{C1}{A1} \ncarc{->}{A1}{B1} \ncarc{->}{B1}{B}

\pnode(2,2.5){a2} \pnode(2.5,2.5){b2} \pnode(3.33,2.5){c2} \pnode(5,2.5){d2} \pnode(10,2.5){e2}
\ncarc{->}{e2}{a2} \ncarc{->}{a2}{b2} \ncarc{->}{b2}{c2} \ncarc{->}{c2}{d2} \ncarc{->}{d2}{B}

\pnode (1.67,1.67){aa2} \pnode(2,1.67){a2} \pnode(2.5,1.67){b2} \pnode(3.33,1.67){c2} \pnode(5,1.67){d2} \pnode(10,1.67){e2}
\ncarc{->}{e2}{aa2} \ncarc{aa2}{a2} \ncarc{->}{a2}{b2} \ncarc{->}{b2}{c2} \ncarc{->}{c2}{d2} \ncarc{->}{d2}{B}

\end{pspicture}

\caption{The map  $g_\mset:Y_\mset\to Y_\mset$ from Lemma~\ref{lem:nwnotprod}}
\label{fig:gm}
\end{figure}

Now assume that $f$ is not locally mixing at $x$; we will show that $x$ is not product nonwandering.  Since $f$ is not locally mixing at $x$, there exists a neighborhood $U$ of $x$ such that the set of non-return times, $\{n>0:f^n(U)\cap U = \emptyset\}$, is infinite; let this set be $\mset$.  Consider the map $f\times g_\mset:X\times Y_\mset\to X\times Y_\mset$.  Since $V$ returns to itself under $g_\mset$ only at times in $\mset$, and $U$ returns to itself under $f$ only at times not in $\mset$, the neighborhood $U\times V$ of $(x,y)$ never returns to itself under $f\times g_\mset$.  Thus $(x,y)$ is not nonwandering, and so $x$ is not product nonwandering.

\end{proof}

So, for example, if $f$ is a nontrivial rotation of the circle, then every point is product recurrent but not product nonwandering.  Conversely, in Example~\ref{ex:spiral}, every nonwandering point is product nonwandering, but only the fixed points are product recurrent.

%\bibliography{GRSCR}


\end{document}